\documentclass[graybox]{svmult}

\usepackage{type1cm}        % activate if the above 3 fonts are
\usepackage{makeidx}         % allows index generation
\usepackage{graphicx}        % standard LaTeX graphics tool
\usepackage{multicol}        % used for the two-column index
\usepackage[bottom]{footmisc}% places footnotes at page bottom

\usepackage{newtxtext}       % 
\usepackage{newtxmath}       % selects Times Roman as basic font

\makeindex             % used for the subject index
\begin{document}

\title*{Central Values for Clebsch-Gordan coefficients}
\author{Robert W. Donley, Jr.}
\institute{Robert W. Donley, Jr., \at Queensborough Community College, Bayside, New York, \email{RDonley@qcc.cuny.edu}}
\maketitle

\abstract*{We develop further properties of the matrices $M(m, n, k)$ defined by the author and W. G. Kim in a previous work.  In particular,  we continue an alternative approach to the theory of Clebsch-Gordan coefficients in terms of combinatorics and convex geometry.  New features include a censorship rule for zeros, a sequence of 36-pointed stars of zeros, and another proof of Dixon's Identity.   As a major application, we reinterpret the work of Raynal {\it et al.} on vanishing Clebsch-Gordan coefficients as a ``middle-out'' approach to computing $M(m, n, k).$ }

\abstract{We develop further properties of the matrices $M(m, n, k)$ defined by the author and W. G. Kim in a previous work.  In particular,  we continue an alternative approach to the theory of Clebsch-Gordan coefficients in terms of combinatorics and convex geometry.  New features include a censorship rule for zeros, a sequence of 36-pointed stars of zeros, and another proof of Dixon's Identity.   As a major application, we reinterpret the work of Raynal {\it et al.} on vanishing Clebsch-Gordan coefficients as a ``middle-out" approach to computing $M(m, n, k).$ }

\section{Introduction}
\label{sec:1}

In the representation theory of $SU(2)$, the Clebsch-Gordan decomposition for tensor products of irreducible representations yields a uniform pattern for highest weights,  generally, as an arithmetic progression of integers with difference two, symmetric about zero.  Curiously, at the vector level, if one tensors two vectors of weight zero, a similar arithmetic progression of weights occurs, but now with difference four.  In the theory of spherical varieties, extensions of this problem consider minimal gap lengths in the weight monoid associated to spherical vector products.   An elementary calculation for the case of $SU(2)$ initiates the present work, given in Proposition 15, and we review the open problem of vanishing beyond the weight zero case. 

Continuing the work began in \cite{Do},  this approach to Clebsch-Gordan coefficients substitutes the use of hypergeometric series \cite{Vi} and the like with elementary combinatorial methods (generating functions, recurrences, finite symmetry groups, Pascal's triangle).  Of course, hypergeometric series are fundamental to the theory; a long range goal would be to return this alternative approach, once sufficiently developed, to the hypergeometric context with general parameters.  At the practical level, certain computer simulations in nuclear physics and chemistry may require excessively large numbers of  Clebsch-Gordan coefficients, possibly with large parameters.  Methods for an integral theory have immediate scientific applications. 

As with \cite{Do}, this work contains many direct elementary proofs of known, but perhaps lesser known at large,  results and gives them a new spatial context. One central item of concern, the domain space, is a set of integer points inside of a five-dimensional cone, equipped with an order 72 automorphism group, which in turn consists of the familiar symmetries for the determinant.

Key observations in this work depend on the fixed points of the symmetry group in the cone, a distinguished subgroup of dihedral type $D_{12}$,  polygonal subsets of the domain invariant under the subgroup action, and the simultaneous use of recurrences with the subgroup's center.  One byproduct of the theory is yet another proof of Dixon's Identity and some variants, and, with this identity, our computational viewpoint shifts from the specific ``outside-in" algorithm of \cite{Do} to a universal ``middle-out" approach, adapted from the work in \cite{RRV}.

Section 2 recalls the algorithm of \cite{Do} for $M(m, n, k)$, and section 3 develops basic properties of $M(m, n, k).$  Sections 4 and 5 review the theory of the so-called ``trivial" zeros, and sections 6--9 reconsider the results of \cite{RRV} as computations near the center of $M(m, n, k).$

\section{Coordinate Vector Matrices for Clebsch-Gordan Sums}
\label{sec:2}

As a function in five variables, $c_{m, n, k}(i, j)$ is defined on the integer points of the cone
\begin{equation} 0\le i \le m,\quad  0\le j\le n,\quad 0\le k \le \min(m, n), \quad 0 \le i+j-k \le m+n-2k.\notag\end{equation}  It is convention to extend this domain  by zero; here it is enough to do so at least for the matrices $M(m, n, k)$ defined below and for clarity we often omit the corner zeros.
In general, the various Clebsch-Gordan coefficients $C_{m, n, k}(i, j)$ and $c_{m, n, k}(i, j)$ differ by  a nonzero factor, and our approach to vanishing of Clebsch-Gordan coefficients is through vanishing of the sum $c_{m, n, k}(i, j).$  

From \cite{Do}, all $c_{m, n, k}(i, j)$  in (3) below may be computed algorithmically as a matrix $M(m, n, k)$.  
With $m, n, k$ fixed, Proposition 1 and Theorem 1 below produce $M(m, n, k),$ with columns corresponding to coordinate vectors for weight vectors in the subrepresentation $V(m+n-2k)$ of $V(m)\otimes V(n)$.  Consideration of Pascal's identity gives an explicit formula for $c_{m, n,k}(i, j)$ in Proposition 2.
 
For non-negative integers $a, b , c$, multinomial coefficients are defined by
\begin{equation}\left(\begin{array}{c} a+b \\ a\end{array}\right) = \frac{(a+b)!}{a!\ b!}\qquad\mbox{and}\qquad\left(\begin{array}{c} a+b+c \\ a,\ b,\ c\end{array}\right) = \frac{(a+b+c)!}{a!\ b!\ c!}.\end{equation}
First the highest weight vectors for $V(m+n-2k)$ are defined by
\begin{proposition}[Leftmost column for $M(m, n, k)$]  With $0\le i\le k$, define the $(i+1, 1)$-th entry of $M(m, n, k)$  by
\begin{equation}c_{m, n, k}(i, k-i) = (-1)^i\left(\begin{array}{c}m-i \\ k-i \end{array}\right)\left(\begin{array}{c}n-k+i \\ i \end{array}\right).\end{equation}
\end{proposition}
Repeated application $f$ to these highest weight vectors produces coordinates for general weight vectors in the corresponding $V(m+n-2k),$ recorded as
\begin{proposition}[Entry at coordinate $(i+1, i+j-k+1)$ in $M(m, n, k)$]
\begin{equation}c_{m,n,k}(i,j) = \sum_{l=0}^{k} (-1)^l  \left(\begin{array}{c}i+j-k \\ i-l \end{array}\right)\left(\begin{array}{c}m-l \\ k-l \end{array}\right)\left(\begin{array}{c}n-k+l \\ l \end{array}\right).\end{equation}
\end{proposition}

We also note the following alternative expression from \cite{Do}:
\begin{equation}c_{m,n,k}(i,j) = \sum_{l=0}^{k} (-1)^l  \left(\begin{array}{c}i+j-k \\ i-l \end{array}\right)\left(\begin{array}{c}m-i\\ k-l \end{array}\right)\left(\begin{array}{c}n-j \\ l \end{array}\right).\end{equation}
With this expression, results in later sections may be interpreted by way of Dixon's Identity and its extensions for binomial sums.

\begin{theorem}[Definition of Matrix $M(m, n, k)$]  To calculate the coordinate vector matrix  $M(m, n, k)$:
\begin{enumerate}
\item initialize a matrix with $m+1$ rows and $m+n-2k+1$ columns,
\item set up coordinates for the highest weight vector in the leftmost column using Proposition 1, and extend the top row value,
\item apply Pascal's recurrence rightwards in an uppercase L pattern, extending by zero where necessary,
\item for the zero entries in lower-left corner, corresponding entries in the upper-right corner are set  to zero, and
\item the $(i+1, i+j-k+1)$-th entry is $c_{m, n, k}(i, j).$
\end{enumerate}
\end{theorem}
In this work, we often remove corner zeros for visual clarity.  Many examples of $M(m, n, k)$ will be given throughout this work.

\section{Elementary Rules}
\label{sec:3}

A useful parametrization for the domain of $c_{m, n, k}(i, j)$ is given by the set of Regge symbols
\begin{equation}\left|\left|\begin{array}{ccc} \phantom{-}n-k & \phantom{-}m-k & \phantom{-}k \\ \phantom{-}i &\phantom{-} j &\phantom{-} m+n-i-j-k\ \\ \phantom{-}m-i & \phantom{-}n-j & \phantom{-}i+j-k\end{array}\right|\right|.\end{equation}  Note that each row or column sums to $J=m+n-k.$
That is, the domain space is  in one-one correspondence with all 3-by-3 matrices with nonnegative integer entries having this magic square property.  Here we follow \cite{Vi}, while \cite{RRV} switches rows 2 and 3.

In turn, one may define the Regge group of symmetries; each of the 72 determinant symmetries,  generated by row exchange, column exchange, and transpose, correspond to a transformation of $c_{m, n, k}(i, j).$ 

Of particular interest here is the dihedral subgroup $D_{12}$ of twelve elements generated by column switches and the interchange of rows 2 and 3.  Some elements, including a generating set, are given as follows:

\begin{proposition}[R23 Symmetry - Weyl Group] With the asterisks defined by (1),
\begin{equation}\left(\begin{array}{c}m+n-k \\ i,\ j,\ * \end{array}\right)  c_{m, n, k}(m-i, n-j) = (-1)^k \left(\begin{array}{c}m+n-k \\ m-i,\ n-j,\ * \end{array}\right) c_{m, n, k}(i, j).\end{equation}
\end{proposition}

\begin{proposition}[C12 Symmetry] When defined, 
\begin{equation}c_{n,m,k}(j, i) = (-1)^kc_{m, n, k}(i, j).\end{equation}
\end{proposition}

\begin{proposition}[C13 Symmetry]  With $i'=i+j-k$ and $m'=m+n-2k$,
\begin{equation}c_{m', n, n-k}(m'-i', j) = (-1)^{n-j}c_{m, n, k}(i, j).\end{equation}
\end{proposition}

\begin{proposition}[C123 Symmetry] With $i'=i+j-k$ and $m'=m+n-2k$,
\begin{equation}c_{m', m, m-k}(m'-i', i) = (-1)^{m-k+i}c_{m, n, k}(i, j).\end{equation}
\end{proposition}

Since the Weyl group symmetry preserves $m$, $n$, and $k$, it transforms $M(m, n, k)$ to itself. In fact, the net effect of  this symmetry is to rotate $M(m, n, k)$ by 180 degrees, change signs according to the parity of $k$, and rescale values by a positive scalar, rational in the five parameters. In particular, the zero locus of $c_{m, n, k}(i, j)$ in $M(m, n, k)$ is preserved under this symmetry.

In the complement of the corner triangles of zeros, the upper, left-most, and lower-left edges in $M(m, n, k)$ have non-vanishing entries by construction.  By the Weyl group symmetry, the remaining outer edges also have this property. Thus these edges trace out a polygon, possibly degenerating to a segment, with no zeros on its outer edges.  
\begin{definition} We refer to the complement of the corner triangles of zeros in $M(m, n, k)$ as the {\bf polygon} of $M(m, n, k)$.  A zero in the interior of the polygon of $M(m, n, k)$ is called a {\bf proper zero}. Other terminology for zeros will be noted below. \end{definition} 

Opposing edges of this polygon have the same length. The horizontal edges have length $n-k+1$, the slanted edges have length $m-k+1$, and the vertical edges have length $k+1$.  When the polygon is a hexagon, one notes that the absolute values of the entries along the top, lower-left, and right-sided edges are constant and equal to, respectively, 
\begin{equation} \left(\begin{array}{c}m \\ k \end{array}\right),\ \left(\begin{array}{c} n \\ k \end{array}\right),\  \mbox{and}\ \left(\begin{array}{c}m+n-2k \\ m-k \end{array}\right).
\end{equation}
These binomial coefficients are composed of parts $m-k$, $n-k,$ and $k.$ 
These values apply to the degenerate cases (parallelogram or line segment) accordingly. The remaining edges link these values through products of binomial coefficients; the indices in the starting coefficient decrease by 1 as the indices in the terminal coefficient increase likewise. See Proposition 1 for the leftmost vertical edge.

We note the decomposition
\begin{equation}D_{12}\cong S_3\times C_2,\end{equation}
where $S_3$ represents the permutation subgroup generated by column switches and the Weyl group symmetry generates the two-element group $C_2$.  While column switches do not preserve $M(m, n, k)$ in general,  the polygon of $M(m, n, k)$ maps to the polygon of values in another $M(m_2, n_2, k_2)$, differing only by sign changes.  Thus there is a well-defined correspondence between the proper zeros of $M(m, n, k)$ and $M(m_2, n_2, k_2)$ under a column switch.

Of the column switches listed,  the C12 symmetry changes sign according to $k$ and inverts proper values in each column, and the C13 symmetry changes sign according to $n-j$ and reflects values across the northeasterly diagonal.  The column switch C123 rotates values by 120 degrees counter-clockwise and changes sign according to $m-k+i.$  

For example, the polygons of $M(3, 4, 2), M(3, 3, 1), M(4, 3, 2)$ below permute among themselves under the $D_{12}$ symmetries:

\begin{equation}
\left[
\begin{array}{cccc}
\phantom{-} 3  & \phantom{-}3  & \phantom{-}3   & \\
 -6 & -3 & \phantom{-} 0  & \phantom{-}3  \\
\phantom{-} 6 &\phantom{-} 0     & -3  & -3 \\
  & \phantom{-}6  & \phantom{-}6   & \phantom{-}3  \\
\end{array}
\right],\quad
\left[
\begin{array}{ccccc}
 \phantom{-}3  & \phantom{-}3   & \phantom{-}3   &  &  \\
 -3 & \phantom{-}0 &   \phantom{-}3  & \phantom{-}6   &   \\
  & -3     & -3 &\phantom{-}0 & \phantom{-}6 \\
     &    & -3   & -6 & -6 \\
\end{array}
\right],\quad
\left[
\begin{array}{cccc}
 \phantom{-}6  & \phantom{-}6   &    & \\
 -6 & \phantom{-}0 & \phantom{-}6  &    \\
 \phantom{-}3  & -3     & -3  & \phantom{-}3 \\
     & \phantom{-}3   & \phantom{-}0   & -3  \\
     &       & \phantom{-}3   & \phantom{-}3 \\
\end{array}\right].\notag
\end{equation}

Next, we note two additional recurrences from \cite{Do}; these impose further restrictions on proper zeros within $M(m, n, k)$.  

\begin{proposition}[Pascal's Recurrence] When all terms are defined,

\begin{equation}c_{m, n, k}(i, j) = c_{m, n, k}(i, j-1) + c_{m, n, k}(i-1, j).\notag\end{equation}
\end{proposition}

\begin{proposition}[Reverse Recurrence] When all terms are defined,
\begin{equation}\notag a_1\, c_{m, n, k}(i,j)= a_2 \, c_{m, n, k}(i+1, j) + a_3 \, c_{m, n, k}(i, j+1)\end{equation}
where 
\begin{enumerate}
\item $a_1=(i+j-k+1)(m+n-i-j-k)$,
\item $a_2=(i+1)(m-i)$, and
\item $a_3=(j+1)(n-j)$.
\end{enumerate}
\end{proposition}

The reverse recurrence is Pascal's recurrence after application of  the Weyl group symmetry.  It follows a rotated capital-L pattern, only now weighted by positive integers when $0< i<m$ and $0<j<n$.  

These recurrences immediately yield

\begin{proposition}[Censorship Rule]  Suppose $c_{m,n,k}(i, j)=0$ properly in $M(m, n, k)$.  Then adjacent zeros in $M(m, n, k)$  may occur only at the upper-right or lower-left entries. That is, in the following submatrix of $M(m, n, k)$, the bullets must be nonzero:
\begin{equation}\left[\begin{array}{ccc} \bullet & \phantom{-}\bullet & \phantom{-}*\\ \bullet & \phantom{-}0 & \phantom{-}\bullet\\ * & \phantom{-}\bullet & \phantom{-}\bullet \end{array}\right].\notag
\end{equation}
\end{proposition}

\begin{proof} \ \  First note that, in either recurrence relation, if any two terms in the relation are zero, then so is the third.  Now suppose a pair of proper zeros are horizontally adjacent.  Then, alternating the two relations, one begins a sequence
\begin{equation}\left[\begin{array}{ccc} *  & \phantom{-}* & \phantom{-}* \\ * &\phantom{-} *  & \phantom{-}* \\ * & \phantom{-}0 & \phantom{-}0 \end{array}\right] \rightarrow \left[\begin{array}{ccc}  \phantom{-}* &\phantom{-} * & * \\ * &  \phantom{-}0 & \phantom{-}* \\ * & \phantom{-}0 & \phantom{-}0 \end{array}\right]\rightarrow \left[\begin{array}{ccc} * & \phantom{-}* & \phantom{-}* \\ 0  & \phantom{-}0  & \phantom{-}* \\ * & \phantom{-}0 & \phantom{-}0 \end{array}\right]\rightarrow\left[ \begin{array}{ccc} 0 & \phantom{-}* & \phantom{-}* \\ 0  & \phantom{-}0  & \phantom{-}* \\ * & \phantom{-}0 & \phantom{-}0 \end{array}\right]\rightarrow\cdots.\notag\end{equation} Eventually this serpentine must place a zero at a nonzero entry on the top row of $M(m, n, k)$, a contradiction.
The case of two vertically adjacent zeros follows similarly.  Finally, for the diagonal case
\begin{equation}\left[\begin{array}{cc} 0 & \phantom{-}* \\ * & \phantom{-}0 \end{array}\right],\notag\end{equation}
either recurrence rule  begins the serpentine. \ \ $\qed$
\end{proof}

As an adjunct to censorship, certain pairs of zeros severely restrict nearby values. 

\begin{proposition}  With $X$ nonzero, the following allowable pairs of zeros fix adjacent values:
\begin{equation}
\left[\begin{array}{ccc}
 \phantom{-}X & & \\
 \phantom{-}0 & \phantom{-}X & \\
 -X & -X & \phantom{-}0 
\end{array}\right],\quad
\left[\begin{array}{ccc}
 \phantom{-}X& & \\
  -X & \phantom{-}0 & \\
 \phantom{-}0 & -X & -X  
\end{array}\right],\quad
\left[\begin{array}{ccc}
\phantom{-}0 & & \\
 \phantom{-}X & \phantom{-}X & \\
 -X & \phantom{-}0 & \phantom{-}X  
\end{array}\right].\notag
\end{equation}  
In particular, each triangle implies one of the following three equalities:
\begin{equation}(i+1)(m-i) = (i'+1)(m'-i') = (j+1)(n-j),\notag\end{equation}
where $i'=i+j-k$, $m'=m+n-2k$, and $c_{m,n,k}(i, j)$ corresponds to the middle entry of the leftmost column.
\end{proposition}

\begin{proof}\ \   The restriction on values follows immediately from Pascal's recurrence.  The parameter conditions follow from the second recurrence.
$\qed$
\end{proof}

Finally, it will be convenient to note four degenerate cases of $M(m, n, k);$ the first three cases give all conditions for when the polygon is not a hexagon.  

\begin{enumerate}
\item  When $k=0$ and $n > 0,$ the parallelogram of nonzero entries of $M(m, n, 0)$ consists of entries in Pascal's triangle, with columns corresponding to segments of the triangle's rows.
When $n=k=0,$  $M(m, 0, 0)$ is an identity matrix of size $m+1,$
\item When $k=m$ with $n\ge m$, $M(m, n, m)$ contains no zeros; ignoring signs, the upper-right corner corresponds to the peak of Pascal's triangle, with diagonals corresponding to segments of the triangle's rows, 
\item When $m=n=k,$ $M(m, m, m)$ degenerates to a vertical segment of length $m+1$, and
\item When $m\ge 4$ even, $n=2$ and $k=1$, a central vertical triplet of zeros occurs, but only the central zero is proper.
\end{enumerate}

In fact, with $n\ge m$, the C23 symmetry carries $M(m, n-m, 0)$ to  $M(m, n, m)$  

Cases 1--4 are represented below by $M(2, 2, 0),$  $M(2, 4, 2),$  $M(2, 2, 2),$ and $M(4, 2, 1),$ respectively:
\begin{equation}
\left[
\begin{array}{ccccc}
\phantom{-}1  &\phantom{-}1  & \phantom{-}1   & \phantom{-} 0 &\phantom{-} 0\\
 \phantom{-}0 & \phantom{-}1 &  \phantom{-}2  & \phantom{-}3  &  \phantom{-}0\\
 \phantom{-}0 & \phantom{-}0  & \phantom{-}1  & \phantom{-}3 & \phantom{-}6\\
\end{array}\right],\quad
\left[\begin{array}{ccc}
 \phantom{-}1  & \phantom{-}1   &  \phantom{-}1 \\
 -3 & -2 &    -1  \\
\phantom{-}6  & \phantom{-}3     & \phantom{-}1 \\
\end{array}\right],\quad
\left[\begin{array}{c}
\phantom{-}1 \\
 -1   \\
\phantom{-} 1  \\
\end{array}\right],\quad
\left[\begin{array}{ccccc}
\phantom{-} 4  & \phantom{-}4   &  \phantom{-}0  & \phantom{-}0 & \phantom{-}0 \\
 -2 & \phantom{-}2 &\phantom{-} 6  & \phantom{-}0  & \phantom{-}0\\
\phantom{-}0 & -2     & \phantom{-}0  & \phantom{-}6 & \phantom{-}0\\
 \phantom{-}0 &  \phantom{-}0  & -2   & -2 & \phantom{-}4 \\
 \phantom{-}0    & \phantom{-}0  &  \phantom{-}0  & -2 & \phantom{-}4 \\
\end{array}\right].\notag
\end{equation}

%section 4
\section{Diagonal Zeros}
\label{sec:4}

Consideration of dihedral reflections leads one to a large family of proper zeros through Propositions 4, 5, and the C23 symmetry.  In particular, these zeros occur when a column switch in the Regge symbol fixes an entry of $M(m, n, k)$ and changes parity.  Zeros of this type always occur as diagonal subsets in $M(m, n, k)$.

To see this, first observe that when $n=2k$, $M(m, 2k, k)$ is a square matrix of size $m+1$, and the C13 symmetry preserves the northeasterly diagonal.  In this case, the subgroup of the Regge group generated by C13 and R23, of type $C_2 \times C_2,$ preserves both the polygon and the diagonal.  Specifically, the top row of the general Regge symbol for these parameters
\begin{equation}\left|\left|\begin{array}{ccc} k & \phantom{-}m-k & \phantom{-}k \\ i & \phantom{-}j & \phantom{-}m-i-j-k \\ m-i & \phantom{-}2k-j & \phantom{-}i+j-k\end{array}\right|\right|.\end{equation} 
is unchanged by the R23 and C13 symmetries.

With $m\ne 2k$, four edges of the polygon now have length $k+1,$ upon which the subgroup acts transitively.  When $m=2k$, the polygon is a regular hexagon preserved by the dihedral subgroup $D_{12}$ of Section 3.  

\begin{proposition} Suppose $k>0$ and $m+k$ is odd. Then
\begin{equation}c_{m, 2k, k}(i, m+k-2i)\  =\ 0.\end{equation}
\end{proposition}

\begin{proof}\ \ Note that coordinates in $M(m, n, k)$ are indexed by $x=i+j-k+1$ and $y=i+1,$ and  the indices for the diagonal in question are solutions to 
\begin{equation}(i+j-k+1) + (i+1) = m+2 \quad\mbox{or}\quad j= m+k-2i.\end{equation}
Substituting $n=2k$ into Proposition 5, one obtains
\begin{equation}c_{m, 2k, k}(m+k-i-j, j) = (-1)^{j}c_{m, 2k, k}(i, j).\end{equation}
For entries on the diagonal, this equation reduces to
\begin{equation}c_{m, 2k, k}(i, j) = (-1)^{m+k}c_{m, 2k, k}(i, j),\end{equation}
and $c_{m, 2k, k}(i, j)$ vanishes under the given parity condition.
$\qed$
\end{proof}

For example, we have $M(4, 6, 3)$ and $M(5, 4, 2)$, respectively:

\begin{equation}
\left[
\begin{array}{ccccc}
 \phantom{-}4  & \phantom{-}4   &  \phantom{-}4   & \phantom{-}4 &   \\
 -12 & -8 & -4  & \phantom{-}0   & \phantom{-}4 \\
 \phantom{-}20  & \phantom{-}8  &\phantom{-} 0  & -4 & -4 \\
-20    & \phantom{-}0   &  \phantom{-}8   &  \phantom{-}8 &  \phantom{-}4  \\
     & -20     & -20   &  -12 & -4  
\end{array}\right],\quad
\left[\begin{array}{cccccc}
 \phantom{-}10  & \phantom{-}10   & \phantom{-}10   &  &  &  \\
 -12 & -2 & \phantom{-}8  & \phantom{-}18   &  &  \\
  \phantom{-}6 & -6 & -8 & \phantom{-}0 &  \phantom{-}18 &  \\
   & \phantom{-}6  & \phantom{-}0  & -8 & -8 & \phantom{-}10 \\
     &     &  \phantom{-}6   &  \phantom{-}6 &  -2 & -10  \\
     &       &     &  \phantom{-}6 & \phantom{-}12 &  \phantom{-}10  
\end{array}\right].\notag
\end{equation}

Next we give a formula for the entries below the diagonal as single term expressions. There are $k$ proper zeros on the diagonal, and we denote the position of such a zero as measured from lower-left to upper-right. 

\begin{proposition} The value of the entry directly below the $t$-th proper zero on the diagonal is given as a single term expression by
\begin{equation}(-1)^{k+t+1} \left(\begin{array}{c} 2k \\ k \end{array}\right) \left( \begin{array}{c} k-1 \\ t-1 \end{array}\right) \frac{(2k-2t+1)!}{(2k-1)!}\ \frac{(\frac{m+k+1}2)!\ (\frac{m-k-3+2t}2)!}{(\frac{m-k-1}2)!\ (\frac{m+k+3-2t}2)!}.\end{equation}
\end{proposition}

\begin{proof} \ \ The coordinates for the $s$-th proper zero on the diagonal are 
\begin{equation}(i+1, i+j-k+1)=\bigg(\frac{m+k-2s+3}2, \frac{m-k+2s+1}2\bigg)\end{equation}
with
\begin{equation}i=\frac{m+k-2s+1}2,\quad j=2s-1.\end{equation}
 The entry directly below the first zero has value
\begin{equation}c_{m, 2k, k}\bigg(\frac{m+k+1}2, 1\bigg)=(-1)^k\left(\begin{array}{c}2k\\ k\end{array}\right).\notag\end{equation}
The second recurrence allows us to compute values below the diagonal by a sequence of factors: at the $s$-th proper zero, 
\begin{equation} \left[\begin{array}{cc} 0 & \phantom{-}x  \\ y & \phantom{-}y \end{array}\right]\qquad \longrightarrow \qquad x=-\frac{a_2}{a_3}y = -\frac{(m+k-2s+3)(m-k+2s-1)}{8s(2k-2s+1)}\ y.\notag \end{equation}
Thus the value of the entry below the $t$-th zero on the diagonal is
\begin{equation}(-1)^{k+t-1}\left(\begin{array}{c} 2k \\ k \end{array}\right) \prod\limits_{s=1}^{t-1}\frac{(m+k-2s+3)(m-k+2s-1)}{8s(2k-2s+1)}. \end{equation}
Since
\begin{equation}\prod\limits_{s=1}^{t-1}(2N-2s) = \frac{2^{t-1}(N-1)!}{(N-t)!},\qquad \prod\limits_{s=1}^{t-1}(2N+2s) = \frac{2^{t-1}(N+t-1)!}{N!}, \end{equation}
and
\begin{equation}\prod\limits_{s=1}^{t-1}(2N-2s+1) = \frac{(2N-1)! (N-t)!}{2^{t-1}(2N-2t+1)!(N-1)!}, \end{equation}
the result follows by substitution into (20).
$\qed$
\end{proof}

In particular, when $m$ even, $k$ odd and $t=\frac{k+1}2,$ the entry below the central zero is
\begin{equation}(-1)^{\frac{k+1}2}\  \frac{2(k+1)^2}{m(m+2)}\ \left(\begin{array}{c}\frac{m+k+1}2\\[0.5ex] \frac{k+1}2,\ \frac{k+1}2,\ \frac{m-k-1}2 \end{array}\right).\end{equation}
When $m$ is odd, $k$ even and $t=\frac{k}2+1,$ there is a central square
\begin{equation}\left[\begin{array}{cc} X &\phantom{-} 0\\ 0 & \phantom{-}X\end{array}\right]\end{equation}
with 
\begin{equation}X= (-1)^\frac{k}{2}\ \frac{k+2}{m+1}\left( \begin{array}{c}\frac{m+k+1}2\\[0.5ex] \frac{k}2,\ \frac{k+2}2,\ \frac{m-k-1}2 \end{array}\right).\end{equation}

Two other types of zero diagonals may be obtained by applying column switches to (13). Applying the C12 symmetry yields
\begin{proposition} Suppose $k>0$ and $n+k$ is odd. Then
\begin{equation}c_{2k, n, k}(n+k-2j, j)\ =\ 0.\end{equation}
\end{proposition}
This diagonal of zeros connects the midpoints of the horizontal edges of the polygon. By the censorship rule, no zeros on this diagonal are adjacent, and the diagonal is fixed by the C23 symmetry.  Applying the C13 symmetry to the parameters of Proposition 13 gives

\begin{proposition} Suppose $k>0$ is odd. Then
\begin{equation}c_{m, m, k}(i, i)\ =\ 0.\end{equation}
\end{proposition}
This diagonal of zeros connects the midpoints of the vertical edges of the polygon. Again, no zeros on this diagonal are adjacent, and this diagonal is fixed by the C12 symmetry. 

From above, $M(4, 6, 3)$  transforms into $M(6, 4, 3)$ and $M(4, 4, 1),$ respectively:

\begin{equation}
\left[\begin{array}{ccccc}
 \phantom{-}20  & \phantom{-}20   &     &  &   \\
 -20 & \phantom{-}0 & \phantom{-}20  &    &  \\
 \phantom{-}12  & -8  & -8  & \phantom{-}12 &  \\
-4    & \phantom{-}8   &  \phantom{-}0   &  -8 &  \phantom{-}4  \\
     & -4     & \phantom{-}4   & \phantom{-}4 & -4\\
  &  & -4 & \phantom{-}0 &\phantom{-} 4\\
  &  &  &\phantom{-} 4 & -4  
\end{array}\right],\quad
\left[\begin{array}{ccccccc}
 \phantom{-}4  & \phantom{-}4   & \phantom{-}4   & \phantom{-}4 &  &  & \\
 -4 & \phantom{-}0 & \phantom{-}4  & \phantom{-}8   & \phantom{-}12 &  & \\
   & -4 & -4 & \phantom{-}0 & \phantom{-}8 &  \phantom{-}20 & \\
   &   & -4  & -8 & -8 & \phantom{-}0 &  \phantom{-}20 \\
     &     &     &  -4 &  -12 & -20 & -20   
\end{array}\right].\notag
\end{equation}

Next suppose $m=n=2k$ with $k$ odd; see Figure 1 below.  The polygon in $M(2k, 2k, k)$ is now a regular hexagon with sides of length $k+1,$ and entries on three sides have constant absolute value $\left(\begin{array}{c} 2k\\ k\end{array}\right).$ The full $D_{12}$ subgroup preserves $M(2k, 2k, k)$, as the general entry has Regge symbol
\begin{equation}\left|\left|\begin{array}{ccc} k & \phantom{-}k & \phantom{-}k \\ i & \phantom{-}j & \phantom{-}3k-i-j\ \\ 2k-i & \phantom{-}2k-j & \phantom{-}i+j-k\end{array}\right|\right|.\end{equation} In particular, when $i=j=k$, the central entry is a fixed point under the full Regge group. Since $k$ is odd, the polygon now possesses three diagonals of zeros (six-pointed star) with several interesting consequences; first, when $k\ge 5$ odd, there is an equilateral triangle, centered about the central value and with sides of length 7, with nonzero entries of fixed absolute value:
\begin{equation}
\left[
\begin{array}{ccccccc}
\phantom{-}0 &  &  &  &  &  & \\
\phantom{-}X   &  \phantom{-}X   &  &  &  &   &  \\
  -X & \phantom{-}{0}  & \phantom{-}\mathbf{X}   &  &   &   & \\
 \phantom{-}{0}  & -X  & \mathbf{-X} & \phantom{-}{0} &  &  & \\
 \phantom{-}\mathbf{X}   &  \phantom{-}\mathbf{X}   &\phantom{-}  \fbox{0} &  \mathbf{-X} & \mathbf{-X} & \mathbf{} &\  \\
 -X  & \phantom{-}{0}    &  \phantom{-}\mathbf{X}   & \phantom{-} X &\phantom{-} {0} & -X &\  \\
 \phantom{-}0  & -X    & \mathbf{-X}   & \phantom{-} {0} &\phantom{-} X & \phantom{-}X & \phantom{-}0 \\
\end{array}\right].\notag
\end{equation} 
By (23), the nonzero entry $X$ is equal to 
\begin{equation}c_{2k, 2k, k}(k+1, k-1) = (-1)^{\frac{k+1}2}\  \frac{k+1}{6k}\ \left(\begin{array}{c}\frac{3k+3}2\\[0.5ex] \frac{k+1}2,\ \frac{k+1}2,\ \frac{k+1}2 \end{array}\right).\notag\end{equation}
Furthermore, the large triangle is an aggregation of the three smaller triangular zero pair patterns in Proposition 10.

Next, as each non-central zero on a diagonal is fixed by an order two subgroup in the Regge group, each orbit under the full Regge group contains 36 zeros. 
Since each Regge symmetry induces a linear change in indices, we have
\begin{theorem} For each odd $k>1$,  $c_{2k, 2k, k}(k, k)$ is a fixed point of the full Regge group and is at the center of both $M(2k, 2k, k)$ and a 36-pointed star of zeros in the five-dimensional Clebsch-Gordan domain space.
\end{theorem}

\begin{figure}
\sidecaption[t]
\includegraphics[scale=.5]{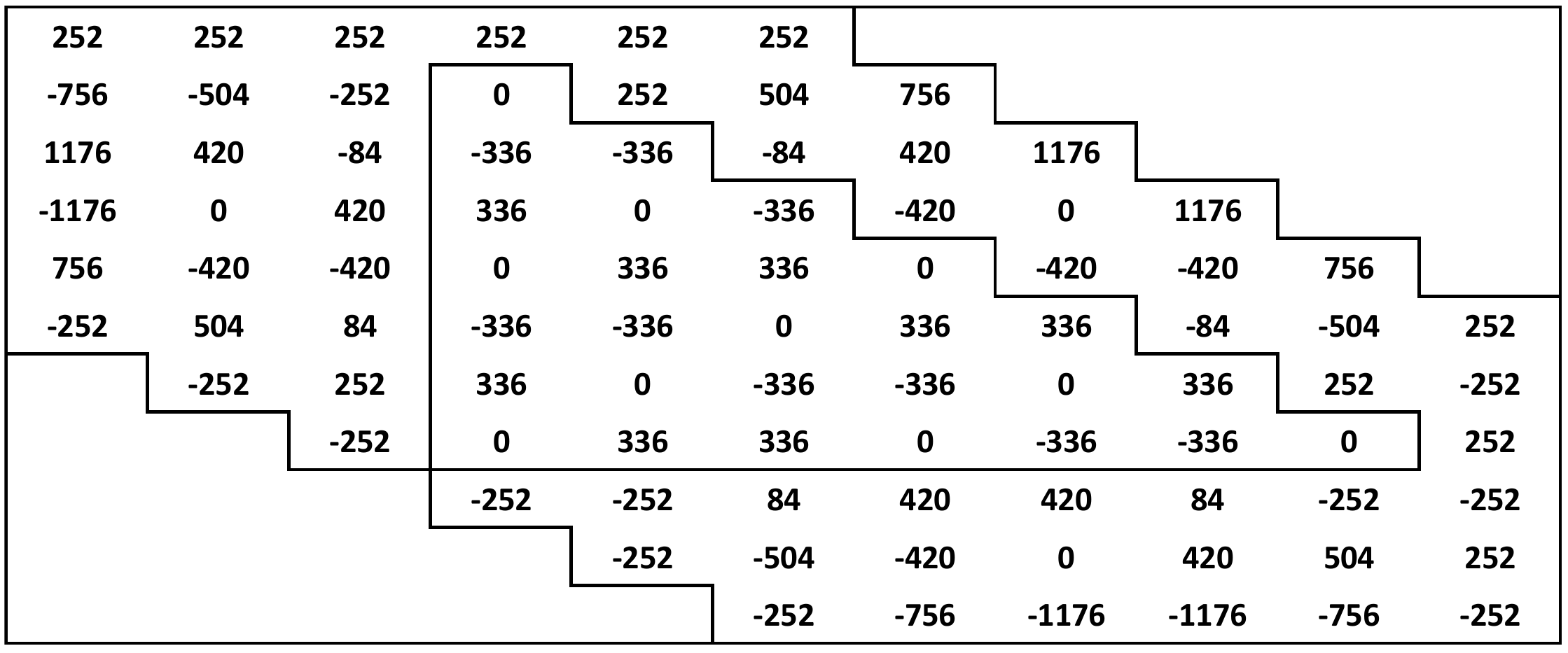}
\caption{M(10, 10, 5) with $m=n=2k$ and $k=5$ odd}
\label{M(10,10, 5)}
\end{figure}

%section 5
\section{Four cases and some special values}
\label{sec:5}

Following \cite{RRV}, we now focus on behavior of $c_{m, n, k}(i, j)$ near the center of $M(m, n, k).$  The shape of the center is determined by parities of $m,$ $n,$ and $k$; these are determined by the top line of the Regge symbol, and, through the $D_{12}$ symmetries, we can narrow our results  to four cases: 

\begin{table}
\caption{Central area of $M(m, n, k)$ based on parity}
\label{tab:1}
\begin{tabular}{p{1.8cm}p{1.8cm}p{1.8cm}p{3cm}}
\hline\noalign{\smallskip}
  $m-k$ & $n-k$ & $k$ & Center \\
 \noalign{\smallskip}\svhline\noalign{\smallskip}
Even & even & even & single entry ($\ne 0$)\\
Odd & even & even & size 2 square  \\
Even & odd & odd & size 2 square\\
Odd & odd  & odd & single entry ($=0$)\\
\noalign{\smallskip}\hline\noalign{\smallskip}
\end{tabular}
\end{table}

Suppose both $m$ and $n$ are even.  Recall that $M(m, n, k)$ is a matrix of size $m+1$ by $m+n-2k+1.$  Thus there is a central entry, $c_{m, n, k}(\frac{m}2, \frac{n}2)$, which we refer to as the {\bf central value} of $M(m, n, k)$, and, with its adjacent entries, we obtain a square submatrix of size 3, which we refer to as the {\bf central square} of $M(m, n, k).$  As seen in the fourth degenerate case, a central square for $k>0$ can only have non-proper zeros when  $n=2$, $k=1,$ and $m\ge 4$ is even.

\begin{proposition}
Suppose $m$ and $n$ are even. The central value $c_{m, n, k}(\frac{m}2,\frac{n}2)=0$ if and only if $k$ is odd.
\end{proposition}

\begin{proof} 
If $k$ is odd then the Weyl group symmetry implies
\begin{equation}c_{m, n, k}(\frac{m}2, \frac{n}2) = -c_{m, n, k}(\frac{m}2, \frac{n}2),\end{equation}
and the central value vanishes.

In the other direction, suppose the central value vanishes. Consider the central square
\begin{equation}\left[\begin{array}{ccc} -Y  & \phantom{-}* & \phantom{-}* \\  \phantom{-}Y & \phantom{-}0  & \phantom{-}* \\  \phantom{-}* & \phantom{-}X & \phantom{-}X \end{array}\right]\notag\end{equation}
with $X$ and $Y$ nonzero.
From the lower left hook and the second recurrence rule with $i=\frac{m}2$ and $j=\frac{n}2-1$, we have
\begin{equation}a_1Y =a_2X,\notag\end{equation}
and positivity of $a_i$ implies that $X$ and $Y$ have the same parity.  Since $-Y$ and $X$ have opposite parity, the Weyl group symmetry switches signs and $k$ is odd.
$\qed$
\end{proof}

\begin{remark}   This condition is equivalent to  the Regge symbol having  matching bottom rows with  $J=m+n-k$ odd. This result also corresponds to the linearization formula for products of Legendre polynomials, as seen, for instance, in Corollary 6.8.3 in \cite{AAR}, and it may be shown directly using the Weyl group and the Casimir operator.  In the physics literature,  zeros of this type, or their translates under the Regge group, are referred to as ``trivial" zeros; all diagonal zeros are of this type.  Trivial zeros also correspond to indices outside the polygons and those omitted under raising or lowering operators.

In turn, the proposition may be interpreted as a ``gap 4" result when tensoring vectors of weight zero, in contrast with the usual Clebsch-Gordan decomposition, which corresponds to ``gap 2."  Since a contribution only occurs for $k$ even,
\begin{equation}f^{m/2}\phi_m\otimes f^{n/2}\phi_n = \sum\limits_{k'=0}^{min(m, n)/2} C_{m, n, 2k'}(m/2, n/2) f^{m/2+n/2-2k'} \phi_{m, n, 2k'}.\end{equation}
Note that the vectors in the sum have nonzero coefficients and correspond to irreducible constituents $V(m+n-4k')$ for $0\le k' \le \frac12\min(m, n).$ 
\end{remark}

A first step to computing near the center of $M(m, n, k)$ requires knowing either the central value or a near central value.  To compute these values inductively, one first notes some values of $c_{m, n, k}(i, j)$ for small $k$ and a four-term recurrence relation. One obtains the following directly from either summation formula:
\begin{equation}c_{m, n, 0}(i, j) = \left(\begin{array}{c} i+j \\ i\end{array}\right),\quad c_{m, n, 1}(i, j) = \left(\begin{array}{c} i+j \\ i\end{array}\right)\frac{mj-ni}{i+j},\end{equation}
\begin{equation}c_{m, n, 2}(i, j) = \left(\begin{array}{c} i+j \\ i\end{array}\right)\ \frac{jm(j-1)(m-1)-2ij(m-1)(n-1) + in(i-1)(n-1)}{(i+j)(i+j-1)}\end{equation}
Zeros corresponding to $k=1, 2$ are further classified in \cite{Va} and \cite{Lo}, respectively.

Next
\begin{lemma} When all terms are defined,
\begin{equation}c_{m+2, n+2, k+2}(i+1, j+1) + c_{m, n, k}(i, j)
= c_{m, n+2, k+2}(i, j+1) + c_{m+2, n, k+2}(i+1, j).\notag
\end{equation}
\end{lemma} 

\begin{proof}\ \ Using formulas (7.2) through (7.5) in   \cite{Do}, we have
\begin{eqnarray}&&\hspace{-30pt}c_{m+2, n+2, k+2}(i+1, j+1)\notag\\
&=& c_{m+1, n+2, k+2}(i, j+1)+c_{m+2, n+1, k+2}(i+1, j)\notag\\[1.0ex]
&=& c_{m, n+2, k+2}(i, j+1) + c_{m, n+1, k+1}(i, j)\notag\\ 
&&\qquad\qquad\qquad\qquad + c_{m+2, n, k+2}(i+1, j) -c_{m+1, n, k+1}(i, j)\notag\\[1.0ex]
&=& c_{m, n+2, k+2}(i, j+1) + c_{m+2, n, k+2}(i+1, j) - c_{m, n, k}(i, j). \notag
\end{eqnarray} 
\end{proof}

See \cite{Ra} for a normalized version of the lemma, along with normalized versions of \cite{Do}, (7.2), (7.3).  As a special case of the lemma, we obtain another proof of Dixon's Identity, first proven in \cite{Fj}. Many alternative proofs exist; see, for instance, \cite{GS}, \cite{Ek}, and \cite{Wa}.

\begin{theorem}[Dixon's Identity] When $m, n,$ and $k$ are even, the central value
\begin{equation}c_{m, n, k}(\frac{m}2, \frac{n}2)\ =\sum_{l=0}^{k} (-1)^l  \left(\begin{array}{c}\frac{m+n-2k}2 \\ \frac{m}2-l \end{array}\right)\left(\begin{array}{c}\frac{m}2 \\ k-l \end{array}\right)\left(\begin{array}{c}\frac{n}2 \\ l \end{array}\right)=\ (-1)^\frac{k}2 \left(\begin{array}{c} \frac{m+n-k}2\\ \frac{m-k}2,\ \frac{n-k}2,\ \frac{k}2\end{array}\right).\notag\end{equation}
\end{theorem}

\begin{proof}\ \   We prove the theorem by induction on $N=m+n+k$. For the base case, if $k=0$,  the result follows by (31).  Now suppose the theorem holds for $N\le m+n+k+4$.   Using Lemma 1,

\begin{eqnarray}
&&\hspace{-30pt}c_{m+2, n+2, k+2}\bigg(\frac{m+2}{2}, \frac{n+2}2\bigg)\notag\\
&=& c_{m, n+2, k+2}\bigg(\frac{m}2, \frac{n+2}2\bigg) + c_{m+2, n, k+2}\bigg(\frac{m+2}2, \frac{n}2\bigg) - c_{m, n, k}\bigg(\frac{m}2, \frac{n}2\bigg)\notag\\[1.0ex]
&=&(-1)^{\frac{k+2}2}\bigg(\left(\begin{array}{c}\frac{m+n-k}2\\[0.5ex] \frac{m-k-2}2, \frac{n-k}2, \frac{k+2}2 \end{array}\right)+\left(\begin{array}{c}\frac{m+n-k}2\\[0.5ex] \frac{m-k}2, \frac{n-k-2}2, \frac{k+2}2 \end{array}\right)+\left(\begin{array}{c}\frac{m+n-k}2\\[0.5ex] \frac{m-k}2, \frac{n-k}2, \frac{k}2 \end{array}\right)\bigg)\notag\\[1.0ex]
&=&(-1)^{\frac{k+2}2}\ \frac{m+n-k+2}{k+2}\left(\begin{array}{c}\frac{m+n-k}2\\[0.5ex] \frac{m-k}2, \frac{n-k}2, \frac{k}2 \end{array}\right)\notag\\[1.0ex]
&=&(-1)^{\frac{k+2}2}\left(\begin{array}{c}\frac{m+n-k+2}2\\[0.5ex] \frac{m-k}2, \frac{n-k}2, \frac{k+2}2 \end{array}\right).\notag
\end{eqnarray}
Thus the induction step holds and the theorem is proved.  $\qed$
\end{proof}

With similar proofs, the remaining three cases follow:

\begin{proposition} With $m$ and $n$ even and $k$ odd, the value below the central zero is given by
\begin{equation}c_{m, n, k}(\frac{m}2+1, \frac{n}2-1)=(-1)^\frac{k+1}2\ \frac{2(k+1)(m-k+1)}{m(m+2)} \left(\begin{array}{c} \frac{m+n-k+1}2\\ \frac{m-k+1}2, \frac{n-k-1}2, \frac{k+1}2 \end{array}\right).\notag\end{equation}
\end{proposition}

\begin{proposition} With $m$ odd and  $n$ and  $k$ even, the lower-right central value is given by
\begin{equation}c_{m, n, k}(\frac{m+1}2, \frac{n}2)=(-1)^\frac{k}2\ \frac{m-k+1}{m+1} \left(\begin{array}{c} \frac{m+n-k+1}2\\[0.5ex] \frac{m-k+1}2, \frac{n-k}2, \frac{k}2 \end{array}\right).\notag\end{equation}
\end{proposition}

\begin{proposition} With $m$ and $k$ odd and  $n$ even, the lower-right central value is given by
\begin{equation}c_{m, n, k}(\frac{m+1}2, \frac{n}2)=(-1)^\frac{k+1}2\ \frac{k+1}{m+1} \left(\begin{array}{c} \frac{m+n-k}2\\[0.5ex] \frac{m-k}2, \frac{n-k-1}2, \frac{k+1}2 \end{array}\right).\notag\end{equation}
\end{proposition}

\section{Case 1: $m$, $n$, and $k$ even}
\label{sec:6}

We now turn our attention to the first of four cases of  ``non-trivial" zeros.  Non-trivial zeros are also called ``polynomial" or ``structural" in the literature.

 For fixed $m$, $n$, and $k,$ the use of exponents $i$ and $j$ allow for indexing of $M(m, n, k)$ more or less according to usual matrix notation.  When $m$ and $n$ are both even,  the polygon rotates or reflects about the central value under the $D_{12}$ symmetries; in the other two cases, the central area may change shape.  Positioning the central value as the origin,  we develop two infinite lattices of rational functions of $m$, $n,$ and $k$, one for each  case in the table when $m$ and $n$ are even.  

To compute a given $M(m, n, k),$ one extracts the appropriate polygon from the lattice, evaluates the corresponding rational function at $m$, $n$, and $k$, and rescales by the central or near-central value from Section 5.  An algorithm to construct this lattice follows:
\begin{enumerate}
\item obtain a recursive formula to compute down two columns from the center,
\item use Pascal's recurrence to compute down-and-to-the-right from the center, and
\item apply the $D_{12}$ symmetries to extend to the entire lattice.
\end{enumerate}

First we consider the case with $k$ even.  The central value $X$ is given by Theorem 3. To begin, we have

\begin{proposition} Suppose $m, n,$ and $k$ are even, and $M(m, n, k)$ has central square
\begin{equation}\left[\begin{array}{ccc} Z X  & \phantom{-}* & \phantom{-}* \\  Y X & \phantom{-}\fbox{X}  & \phantom{-}B_0 X \\  * & \phantom{-}A_1 X & \phantom{-}B_1 X \end{array}\right]\notag\end{equation}
for nonzero $X$.  Then
\begin{equation}B_0=\frac{\lambda_{m'}-\lambda_m+\lambda_n}{2\lambda_n}, \quad A_1=\frac{\lambda_{m'}-\lambda_m-\lambda_n}{2\lambda_m},\quad  B_1= \frac{\lambda_{m'}+\lambda_m-\lambda_n}{2\lambda_m},\notag\end{equation}
where $m'=m+n-2k$ and $\lambda_s=s(s+2).$
\end{proposition}

\begin{proof}\ \   Since $k$ is even, $X$ is nonzero.  The following four equations follow from Propositions 7 and 8 and the Weyl group symmetry (6):
\begin{enumerate}
\item $Z +Y=1,$
\item $1+A_1=B_1,$
\item $\lambda_{m'}Y=\lambda_m A_1+\lambda_n,$ and
\item $\lambda_{m'}Z={\lambda_m}B_1.$
\end{enumerate}
These equations reduce to a linear system in $A_1$ and $B_1$ with the above solutions. The reverse recurrence yields the formula for $B_0$.\ \ $\qed$
\end{proof}

Consider the following fourth-quadrant submatrix with $X$ in the central position:
\begin{equation}
\left[\begin{array}{cccc} 
\fbox{$X$} &  & &  \\
A_1 X & \phantom{-}B_1 X&  &   \\
A_2 X& \phantom{-}B_2 X & \phantom{-}C_2 X&   \\
A_3 X& \phantom{-}B_3 X& \phantom{-}C_3 X& \phantom{-}D_3 X\\
\end{array}\right].
\end{equation}

Alternating between the two main recurrences as in Proposition 19 immediately yields
\begin{theorem} Let $X$ be the central value determined by Theorem 3, and define $\lambda_s=s(s+2).$ The first two columns of matrix (33) are computed recursively by
\begin{equation}A_1 = \frac{\lambda_{m'}-\lambda_m-\lambda_n}{2\lambda_m}, \quad B_1 = \frac{\lambda_{m'}+\lambda_m-\lambda_n}{2\lambda_m},\end{equation}
\begin{equation}A_{s+1}=\frac{(\lambda_{m'}-\lambda_m+\lambda_{2s})A_s + (\lambda_{2s-2}-\lambda_{n}) B_s}{\lambda_m-\lambda_{2s}},\end{equation}
\begin{equation}B_{s+1}=\frac{\lambda_{m'}A_s + (\lambda_{2s-2}-\lambda_{n}) B_s}{\lambda_m-\lambda_{2s}}.\end{equation}
\end{theorem}

In the triangle from the first column  to the diagonal, unreduced denominators are equal along rows and increase by a factor of $\lambda_m-\lambda_{2s}$ as we pass from the $s$-th to the $(s+1)$-st row.  That is, the denominator for index $s+1$ equals
\begin{equation}d_{s+1}=2 \prod_{l=0}^s (\lambda_m-\lambda_{2l})=2^{2s+3}\frac{(\frac{m+2s+2}{2})!}{(\frac{m-2s-2}{2})!}.\end{equation}

This implies immediately
\begin{corollary} For $s\ge 1,$ let $N(A_s)$ and $N(B_s)$ be the numerators in the unreduced expressions of $A_s$ and $B_s$, respectively.  Then $N(A_s)$ and $N(B_s)$ are computed recursively by
\begin{equation}N(A_1) = \lambda_{m'}-\lambda_m-\lambda_n, \quad N(B_1) = \lambda_{m'}+\lambda_m-\lambda_n,\end{equation}
\begin{equation}\left[\begin{array}{c} N(A_{s+1}) \\ N(B_{s+1})\end{array}\right]=\left[\begin{array}{cc}\lambda_{m'}-\lambda_m+\lambda_{2s} &\phantom{-} \lambda_{2s-2}-\lambda_n\\ \lambda_{m'} & \phantom{-}\lambda_{2s-2}-\lambda_n \end{array}\right]\left[ \begin{array}{c}N(A_s) \\ N(B_s) \end{array}\right].\end{equation}
\end{corollary} 

To proceed towards the diagonal, for instance,  we have for $s\ge 1,$
\begin{equation}C_{s+1}= B_{s}+B_{s+1}, \qquad N(C_{s+1})=(\lambda_m-\lambda_{2s})N(B_{s})+ N(B_{s+1}).\end{equation}

The denominator is non-vanishing for $s<\frac{m}2$, and vanishing in a coordinate relative to the the central value reduces to solving the corresponding Diophantine equation, say  $N(A_s)=0,$
in $m,$ $n$, and $k.$  

In \cite{RRV},  three families of zeros corresponding to orders 1, 2, and 3, with 6, 12, and 17 subfamilies, respectively,  are classified.  Here order is a measure of ``distance" using 3-term hypergeometric contiguity relations. It may be computed from the Regge symbol directly (section 4 of \cite{RRV}).    Order 1 subfamilies are indexed $I$ through $VI$, and, in particular, these zeros admit a full parameterization, as do subfamilies 2.7 and 2.8.  Each subfamily of order 2 zeros contains infinitely many zeros.  Cardinality in order 3 is an open question, with infinitely many zeros known in types 3.1 and 3.2.  We further note that the conjecture by Brudno in footnote 7 of \cite{HH} is case $I$ of \cite{RRV}. 

For $m, n, k$ even, these subfamilies correspond to positions around the central value as follows:

\begin{equation}
\left[
\begin{array}{ccccccc}
  3.1    &   \phantom{-}3.2    &  \phantom{-}3.2    &   \phantom{-}3.1   &    &      &      \\
  3.2    &  \phantom{-} 2.1    &  \phantom{-} 2.2   & \phantom{-} 2.1    & \phantom{-} 3.2 &      &      \\
  3.2    &  \phantom{-}2.2   & \phantom{-} I    &  \phantom{-} I   &\phantom{-} 2.2 &  \phantom{-}3.2    &      \\     
  3.1  &  \phantom{-}2.1     &\phantom{-}  I   & \phantom{-} \fbox{$\bullet$}   & \phantom{-} I &  \phantom{-}2.1    &  \phantom{-}3.1   \\
      &   \phantom{-}3.2  & \phantom{-}2.2 & \phantom{-} I    & \phantom{-} I    & \phantom{-} 2.2  &  \phantom{-}3.2  \\
      &       &  \phantom{-} 3.2   & \phantom{-} 2.1    &   \phantom{-}2.2 &  \phantom{-}2.1    &  \phantom{-}3.2  \\
      &       &      &  \phantom{-}  3.1  &  \phantom{-}3.2 &   \phantom{-}3.2   &  \phantom{-} 3.1      
 \end{array}\right] \notag
 \end{equation}
 with Diophantine equations in the subcentral triangle given by
 \begin{itemize}
 \item $I$:  $N(A_1)=\lambda_{m'}-\lambda_m-\lambda_n=0,$
 \item 2.1: $N(A_2) = (\lambda_{m'}-\lambda_m+8)(\lambda_{m'}-\lambda_m-\lambda_n)-\lambda_n(\lambda_{m'}+\lambda_{m}-\lambda_n)=0,$
 \item 2.2: $N(B_2)= \lambda_{m'}(\lambda_{m'}-\lambda_m-\lambda_n)-\lambda_n(\lambda_{m'}+\lambda_{m}-\lambda_n)=0,$
 \item 3.1: $N(A_3)= (\lambda_{m'}-\lambda_m+24)N(A_2)-(\lambda_n-8)N(B_2)=0,$
 \item 3.2: $N(B_3)= \lambda_{m'}N(A_2)-(\lambda_n-8)N(B_2)=0.$
 \end{itemize}
Under the $D_{12}$ symmetries, the numerators change by permuting $m'$, $m$, and $n$ and rescaling as in Propositions 3--6.

\begin{figure}
\includegraphics[scale=.8]{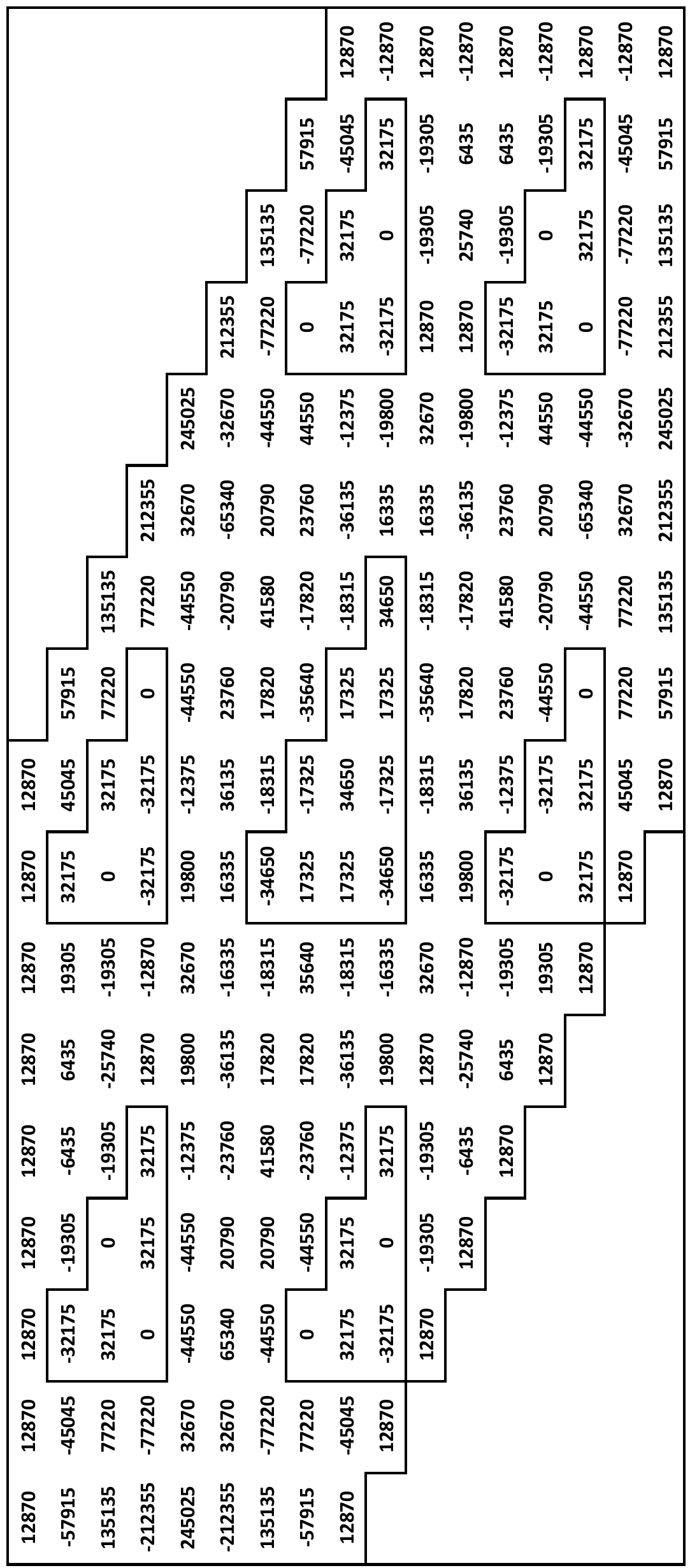}
\caption{M(16, 16, 8): $m=n=2k,$ $k$ even with proper zeros}
\label{M(16, 16, 8)}
\end{figure}

Now suppose $m=n=2k$ with $k$ even. The central value reduces to the original Dixon Identity

\begin{corollary}[Dixon \cite{Di}] When $m=n=2k$ and $k$ even, the central value 
\begin{equation}c_{2k, 2k, k}(k, k)\ = \sum_{l=0}^{k} (-1)^l \left(\begin{array}{c} k \\ l\end{array}\right)^3 =  \ (-1)^\frac{k}2 \left(\begin{array}{c} \frac{3k}2 \\[0.5ex] \frac{k}2,\ \frac{k}2,\ \frac{k}2\end{array}\right).\end{equation}
\end{corollary}

Although we no longer have the diagonals of  zeros from the odd $k$ case in section 4,  there is a central equilateral triangle, with sides of length 4, given by
\begin{equation}
\left[
\begin{array}{cccc}
  -X  & \mathbf{}   &  &   \\[1ex]
   \phantom{-}X/2  & {- X/2} &  &  \\[1ex]
 \phantom{-} {X/2}   &  \phantom{-}\fbox{$X$} &  \phantom{-}{X/2} & \mathbf{}\\[1ex]
   -X    & {-X/2}   & \phantom{-}X/2 & \phantom{-}X 
\end{array}\right].\notag
\end{equation}

We note two conjectures, which hold experimentally for $k<2000:$
\begin{enumerate}
\item when $k$ is odd, zeros only occur on one of the three diagonals in the polygon of $M(2k, 2k, k)$, and
\item when $k$ is even, no zeros occur in the polygon of $M(2k, 2k, k)$ unless $k=8,$ in which case there are six doublets forming a hexagon (Figure 2).  
\end{enumerate}

%section 7
\section{Case 2:  $m$ and $n$ even, $k$ odd}
\label{sec:7}

Assume $m$ and $n$ even, and $k$ odd. This section proceeds in a manner similar, but somewhat simpler than the previous section.  

Consider the following fourth-quadrant submatrix, with central value $0$ and near central value $X$ given by Proposition 16:
\begin{equation}
\left[
\begin{array}{ccccc}
\fbox{$0$} &  & &  \\
X & \phantom{-}X &  &   \\
A_2 X& \phantom{-}B_2 X & \phantom{-}C_2 X&   \\
A_3 X& \phantom{-}B_3 X& \phantom{-}C_3 X& \phantom{-}D_3 X\\
A_4 X & \phantom{-}B_4  X & \phantom{-}C_4 X & \phantom{-}D_4 X & \phantom{-}E_4 X
\end{array}\right].
\end{equation}

Alternating between the two main recurrences as in Proposition 19 immediately yields
\begin{theorem} Let $X$ be the sub-central value determined by Proposition 16, and define $\lambda_s=s(s+2).$ With $s\ge 1,$ the first two columns of (42) are computed recursively by
\begin{equation}A_1 = 1, \quad B_1 = 1,\end{equation}
\begin{equation}A_{s+1}=\frac{(\lambda_{m'}-\lambda_m+\lambda_{2s})A_s - (\lambda_n-\lambda_{2s-2}) B_s}{\lambda_m-\lambda_{2s}},\end{equation}
\begin{equation}B_{s+1}=\frac{\lambda_{m'}A_s - (\lambda_n-\lambda_{2s-2}) B_s}{\lambda_m-\lambda_{2s}}.\end{equation}
\end{theorem}

In the triangle from the first column  to the diagonal, unreduced denominators are equal along rows and increase by a factor of $\lambda_m-\lambda_{2s}$ as we pass from the $s$-th to the $(s+1)$-st row.  That is, with $s\ge 1,$ the denominator for index $s+1$ equals
\begin{equation}d_{s+1}= \prod_{l=1}^s (\lambda_m-\lambda_{2l})=\frac{2^{2s+2}}{\lambda_m} \frac{(\frac{m+2s+2}{2})!}{(\frac{m-2s-2}{2})!}.\end{equation}

This implies immediately
\begin{corollary} For $s\ge 1,$ let $N(A_s)$ and $N(B_s)$ be the numerators in the unreduced expressions of $A_s$ and $B_s$, respectively.  Then$N(A_s)$ and $N(B_s)$ are computed recursively by
\begin{equation}N(A_1) = 1, \quad N(B_1) = 1,\end{equation}
\begin{equation}\left[\begin{array}{c} N(A_{s+1}) \\ N(B_{s+1})\end{array}\right]=\left[\begin{array}{cc}\lambda_{m'}-\lambda_m+\lambda_{2s} &\phantom{-} \lambda_{2s-2}-\lambda_n\\ \lambda_{m'} & \phantom{-}\lambda_{2s-2}-\lambda_n \end{array}\right] \left[\begin{array}{c}N(A_s) \\ N(B_s) \end{array}\right].\end{equation}
\end{corollary} 

To proceed towards the diagonal, for instance,  we have for $s\ge 1,$
\begin{equation}C_{s+1}= B_{s}+B_{s+1}, \qquad N(C_{s+1})=(\lambda_m-\lambda_{2s})N(B_{s})+ N(B_{s+1}).\end{equation}

As before, when $m$ is large enough, the denominator is non-vanishing, and vanishing in a coordinate relative to the the central value reduces to solving the corresponding Diophantine equation, say  
\begin{equation}N(A_s)=0,\end{equation}
in $m,$ $n$, and $k.$  In the classification of \cite{RRV}, diagonal zeros from Section 4 closest to the central zero are denoted by $R$.  For $m, n$ even and  $k$ odd, these subfamilies correspond to positions around the central value as follows:

\begin{equation}
\left[
\begin{array}{ccccccccc}
  3.15    &  \phantom{-} 3.16    &    \phantom{-}3.17  &  \phantom{-}3.16    &   \phantom{-} 3.15 &      &      &       &       \\
  3.16    &  \phantom{-} 2.11    &  \phantom{-}2.12    &  \phantom{-} 2.12   &  \phantom{-}  2.11 & \phantom{-}3.16     &      &       &       \\
  3.17    &   \phantom{-}2.12    &  \phantom{-} II   & \phantom{-} R    & \phantom{-} II &  \phantom{-} 2.12    & \phantom{-}3.17     &       &       \\
  3.16    & \phantom{-} 2.12   &  \phantom{-}R    &  \phantom{-} \bullet   & \phantom{-}\bullet & \phantom{-} R    &  \phantom{-}2.12    &  \phantom{-}3.16     &       \\     
  3.15     &  \phantom{-}2.11  &  \phantom{-}II     &  \phantom{-}\bullet    & \phantom{-} \fbox{0}    & \phantom{-} \bullet &  \phantom{-}II    &  \phantom{-}2.11    &  \phantom{-}3.15  \\
      &\phantom{-}3.16  & \phantom{-}2.12 & \phantom{-} R    &  \phantom{-}\bullet    & \phantom{-} \bullet  &  \phantom{-}R    &  \phantom{-}2.12    &  \phantom{-}3.16    \\
      &       &   \phantom{-}3.17   &  \phantom{-}2.12    &   \phantom{-}II &  \phantom{-}R    &  \phantom{-}II    &  \phantom{-} 2.12    & \phantom{-} 3.17     \\
      &       &      &    \phantom{-}3.16  &  \phantom{-}2.11 &   \phantom{-}2.12   &  \phantom{-} 2.12   &\phantom{-}2.11     &  \phantom{-}3.16     \\          
      &       &      &      &    \phantom{-}3.15 &   \phantom{-}3.16   &   \phantom{-}3.17   &  \phantom{-}3.16     &   \phantom{-}3.15    
 \end{array}\right]. \notag
 \end{equation}

 Diophantine equations in the subcentral triangle are given by
 \begin{itemize}
 \item $II$: $N(A_2)=\lambda_{m'}-\lambda_m-\lambda_n+8=0,$
  \item $R$: $N(B_2)=\lambda_{m'}-\lambda_n=(m-2k)(m+2n-2k+2)=0,$
  \item 2.11: $N(A_3) = (\lambda_{m'}-\lambda_m+24)N(A_2)-(\lambda_n-8)N(B_2)=0,$
  \item 2.12: $N(B_3)= \lambda_{m'}N(A_2)-(\lambda_n-8)N(B_2)=0,$
  \item 3.15:   $N(A_4)= (\lambda_{m'}-\lambda_m+48)N(A_3)-(\lambda_n-24)N(B_3)=0,$
  \item 3.16:   $N(B_4)= \lambda_{m'}N(A_3)-(\lambda_n-24)N(B_3)=0,$
  \item 3.17:   $N(C_4)=(\lambda_m-48)N(B_3)+N(B_4)=0.$
 \end{itemize}
We have relabeled subfamilies 2.15-2.17 in Table 3 of \cite{RRV} as 3.15-3.17 here;  the groupings naturally correspond to concentric hexagons about the center.

\section{Parametrization of Type I and II Zeros}
\label{sec:8}

In \cite{RRV}, a full parametrization for zeros of type $I-VI$ are given.  For expository purposes, we include an algorithm for generating all $(m, n, k)$ satisfying 
\begin{equation}I: N(A_1)=0\qquad \mbox{and}\qquad II: N(A_2)=0.\notag\end{equation}
Types $III-VI$ admit similar parameterizations; each case requires solving a Diophantine equation of the form $xy=uv$, where $x, y, u, v$ are linear expressions in $m, n,$ and $k.$ 

\begin{proposition}  With $k>0$ and even, all solutions to 
\begin{equation}m,\ n\ \mbox{even},\quad m'=m+n-2k,\quad \lambda_{m'}=\lambda_m+\lambda_n\end{equation}
 are given by
\begin{equation}m=2N, \quad n=Q-P-1, \quad k=N-P  \end{equation}
for some integers $N, P, Q$ with
\begin{enumerate}
\item $N\ge 3,$
\item $PQ=N(N+1)$ for $1\le P <Q$ and $P<N,$ and
\item $P$ and $Q$ (resp. $P$ and $N$) have opposite (resp. same) parity.
\end{enumerate}
\end{proposition}

\begin{proof}  Consider the equation \begin{equation}A^2+1= B^2+C^2\end{equation}
with all $A, B, C > 2$ and odd. Basic algebra yields
\begin{equation}\frac{A-C}{2}\, \frac{A+C}{2}=\frac{B-1}{2}\, \frac{B+1}{2}.\end{equation}
Thus solutions are given precisely when
\begin{equation}A=P+Q, \quad B=2N+1,\quad C=Q-P.\end{equation}

Now completing the square in (51) yields
\begin{equation}(m'+1)^2+1= (m+1)^2+(n+1)^2,\end{equation}
and the proposition follows.\ \ $\qed$
\end{proof}

For example, when $N=3,$ $P=1$ and $Q=12$, we have $(m, n, k) = (6, 10, 2).$    See Fig. 3.

\begin{figure}
\includegraphics[scale=.48]{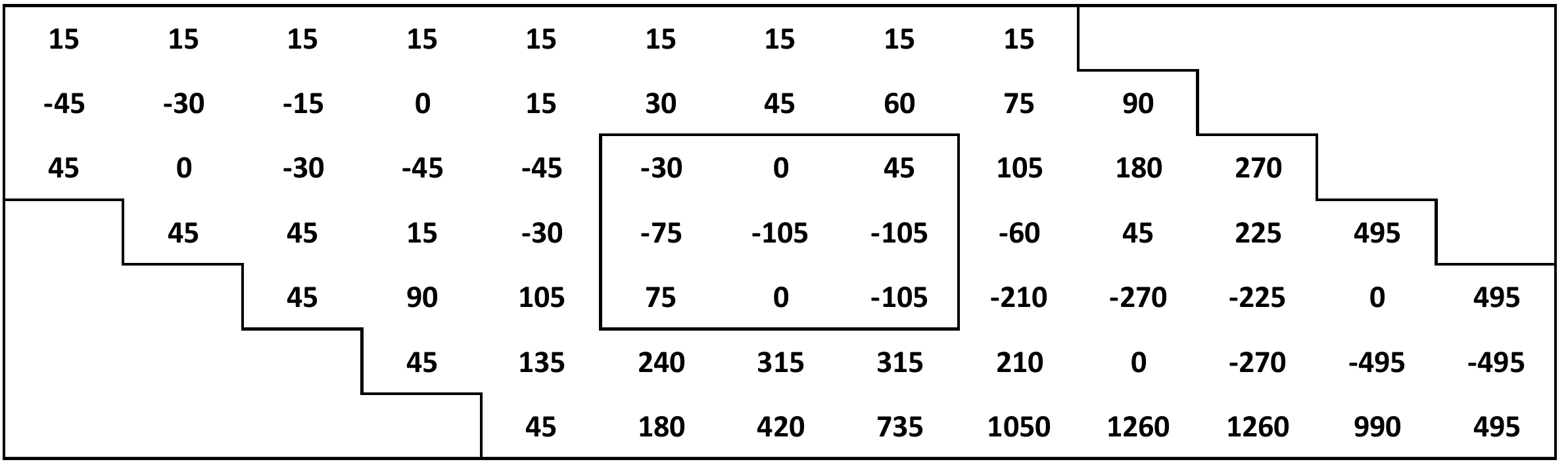}
\caption{$M(6, 10, 2)$: the smallest example with type $I$ zeros}
\end{figure}

Consideration of divisibility properties allows one to directly parametrize some subseries of solutions to (51), as noted in the table below. The first line covers all cases where $P$ divides $N$. The next series gives the general series where the odd part of $P$ divides $N+1$;  in this case, with $0<b, c < 2^{a+1},$
\begin{equation}bc=1\ (mod\ 2^a), \quad bc\ne 1\ (mod\ 2^{a+1}).\end{equation}
Noting that $N$ and $N+1$ have no common factors, we leave it to the reader to generalize to other series.

\begin{table}
\caption{Some Basic Series of Type I Zeros (Position A1)}
\label{tab:2}
\begin{tabular}{p{3cm}p{3cm}p{4cm}}
\hline\noalign{\smallskip}
  $P$ & $N$ & $N+1$ \\
\noalign{\smallskip}\svhline\noalign{\smallskip}  
$s$  & $s(2t+1)$&\\
 $2s+1$ &   & $2t(2s+1)$ \\
 $2(4s+1)$ &  & $(4s+1)(4t+3)$  \\
 $2(4s+3)$ &  & $(4s+3)(4t+1)$ \\
 $\dots$  & & $\dots$\\
 $2^a(2^{a+1}s+b)$ & & $(2^{a+1}s+b)(2^{a+1}t+c)$\\
\noalign{\smallskip}\hline\noalign{\smallskip}
\end{tabular}
\end{table}

Associated to $II: N(A_2)=0$, we have

\begin{proposition}  With $k>1$ and odd, all solutions to 
\begin{equation}m,\ n\ \mbox{even},\quad m'=m+n-2k,\quad \lambda_{m'}+8=\lambda_m+\lambda_n\end{equation}
 are given by
\begin{equation}m=2N+2, \quad n=Q-P-1, \quad k=N-P+1  \end{equation}
for some integers $N, P, Q$ with
\begin{enumerate}
\item $N\ge 3,$
\item $PQ=N(N+3)$ for $1\le P <Q$ and $P<N,$ and
\item $P$ and $Q$ (resp. $P$ and $N$) have opposite (resp. same) parity.
\end{enumerate}
\end{proposition}

\begin{proof}  Completing the square in (58) yields
\begin{equation}(m'+1)^2+9= (m+1)^2+(n+1)^2,\end{equation}
and the proof now follows as in Proposition 20. $\qed$
\end{proof}

\emph{Remark.} For example,  we have $(m, n, k) = (8, 16, 3)$  when $N=3,\ P=1$ and $Q=18$.  

With $l\ge 2,$ solutions of (58) of the form $(m, n, k) = (2l, 2, 1)$ correspond to the fourth degenerate case. That is, with respect to $M(m, n, k)$, we obtain a vertical zero triplet with a single proper zero.

\section{Cases 3 and 4:  $m$ odd, $n$ even}
\label{sec:9}

The remaining two cases allow for a simultaneous treatment.   In both cases, the center is a square of size 2, and the lower-right entries are given by Propositions 17 and 18.  

\begin{proposition} Suppose $m$ is odd and $n$ is even, and $M(m, n, k)$ has central square
\begin{equation}\left[\begin{array}{cc}  X' & * \\ \fbox{$A_0 X$} & \fbox{X}  \end{array}\right]\notag\end{equation}
for nonzero $X$.  Then
\begin{equation}A_0=\frac{m'-m}{m'+1}\ \ \  \mbox{if}\ k\ \mbox{even};\quad A_0=\frac{m'+m+2}{m'+1}\ \  \ \mbox{if}\  k\ \mbox{odd}\notag\end{equation}
where $m'=m+n-2k.$
\end{proposition}

\begin{proof}\ \  See Proposition 19.  In this case, the Weyl group symmetry yields 
\begin{equation}(m+1)X=(-1)^k (m'+1)X'.\end{equation}
$\qed$
\end{proof}

Consider the following fourth-quadrant submatrix, where $X$ represents the lower-right entry of the central square:
\begin{equation}
\left[
\begin{array}{cccc} 
\fbox{$A_0 X$} &\phantom{-} \fbox{X} & &  \\
A_1 X & \phantom{-}B_1 X&  &   \\
A_2 X& \phantom{-}B_2 X & \phantom{-}C_2 X&   \\
A_3 X& \phantom{-}B_3 X& \phantom{-}C_3 X& \phantom{-}D_3\\
\end{array}\right].
\end{equation}
As before, we have 
\begin{theorem} Let $X$ be the lower-right entry determined by Propositions 17 or 18, and define $\lambda_s=s(s+2).$ The first two columns of  (62) are computed recursively by
\begin{equation}A_0=\frac{m'-m}{m'+1}\ \ \ \mbox{if}\ k\ \mbox{even}; \quad A_0=\frac{m'+m+2}{m'+1}\ \ \  \mbox{if} \ k\ \mbox{odd}; \quad B_0 = 1,\end{equation}
\begin{equation}A_{s+1}=\frac{(\lambda_{m'}-\lambda_m+\lambda_{2s+1}+1)A_s + (\lambda_{2s}-\lambda_{n}) B_s}{\lambda_m-\lambda_{2s+1}},\end{equation}
\begin{equation}B_{s+1}=\frac{(\lambda_{m'}+1)A_s + (\lambda_{2s}-\lambda_{n}) B_s}{\lambda_m-\lambda_{2s+1}}.\end{equation}
\end{theorem}

In the subcentral triangle, unreduced denominators are equal along rows and increase by a factor of $\lambda_m-\lambda_{2s+1}$ as we pass from the $s$-th to the $(s+1)$-st row.  That is, with $s\ge 0,$ the denominator for index $s+1$ equals
\begin{equation}d_{s+1}=(m'+1) \prod_{l=0}^s (\lambda_m-\lambda_{2l+1})=2^{2s+3}\ \frac{m'+1}{m+1}\ \frac{(\frac{m+2s+3}{2})!}{(\frac{m-2s-3}{2})!}.\end{equation}

This implies immediately
\begin{corollary} For $s\ge 0,$ let $N(A_s)$ and $N(B_s)$ be the numerators in the unreduced expressions of $A_s$ and $B_s$, respectively.  Then $N(A_s)$ and $N(B_s)$ are computed recursively by
\begin{equation} N(A_0) =  m'-m\ \ \ \mbox{if}\ \  k\ \mbox{even};\quad  N(A_0)= m'+m+2\ \  \ \mbox{if}\ k \ \mbox{odd}; \quad N(B_0) = m'+1, \end{equation}
\begin{equation}\left[\begin{array}{c} N(A_{s+1}) \\ N(B_{s+1})\end{array}\right]=\left[\begin{array}{cc}\lambda_{m'}-\lambda_m+\lambda_{2s+1}+1 &\phantom{-} \lambda_{2s}-\lambda_n\\ \lambda_{m'}+1 & \phantom{-}\lambda_{2s}-\lambda_n \end{array}\right] \left[\begin{array}{c}N(A_s) \\ N(B_s) \end{array}\right].\end{equation}
\end{corollary} 

As before, numerators correspond to the following positions, up to a $C_2\times C_2$ symmetry:

\begin{equation}
\begin{array}{ccc}
& m\ \mbox{odd},\ n\ \mbox{even},\ k\ \mbox{even} & \\[1.0ex]
& \left[\begin{array}{cccccccc}
3.4    & \phantom{-}3.6  & \phantom{-}3.14  & \phantom{-}3.12  &   &   &   & \\
3.6   & \phantom{-}2.4  &  \phantom{-}2.6 &\phantom{-} 2.10  & \phantom{-}3.8  &   &   & \\
3.14   & \phantom{-}2.6  & \phantom{-}IV  & \phantom{-}VI  & \phantom{-}2.8  & \phantom{-}3.10  &   & \\
3.12    & \phantom{-}2.10  & \phantom{-}VI  & \phantom{-}\bullet  & \phantom{-}R  & \phantom{-}2.8  & \phantom{-}3.8  & \\
   & \phantom{-}3.8  & \phantom{-}2.8  & \phantom{-}\fbox{$R$}  & \phantom{-}\fbox{$\bullet$}  & \phantom{-}VI  & \phantom{-}2.10  & \phantom{-}3.12\\
    &   & \phantom{-}3.10  & \phantom{-}2.8  & \phantom{-}VI  & \phantom{-}IV  & \phantom{-}2.6  & \phantom{-}3.14\\
   &   &   & \phantom{-}3.8  & \phantom{-}2.10  & \phantom{-}2.6  & \phantom{-}2.4  & \phantom{-}3.6 \\
    &   &   &   & \phantom{-}3.12  & \phantom{-}3.14  & \phantom{-}3.6  & \phantom{-}3.4\\
\end{array}\right] & 
\end{array},\notag
 \end{equation}

 \begin{itemize}
 \item $R$:  $N(A_0)=m'-m=n-2k=0,$
 \item 2.8: $N(A_1)=(m'-m)(\lambda_{m'}-\lambda_m+4)-\lambda_n(m'+1)=0,$
 \item $VI$: $N(B_1) = (m'+1)[(m'-m)(m'+1)-\lambda_n]=0,$
 \item 3.8: $N(A_2)= (\lambda_{m'}-\lambda_m+16)N(A_1)-(\lambda_n-8)N(B_1)=0,$
 \item 2.10: $N(B_2)= (\lambda_{m'}+1)N(A_1)-(\lambda_n-8)N(B_1)=0,$
 \item 3.12: $N(B_3)=(\lambda_{m'}+1)N(A_2)-(\lambda_n-24)N(B_2)=0,$
 \item IV:  $N(C_1)= (m'+1)[\lambda_m-\lambda_n-3+(m'-m)(m'+1)]=0,$
 \item 2.6: $N(C_2)=(\lambda_m-15)N(B_1)+N(B_2)=0,$
 \item 3.14: $N(C_3)=(\lambda_m-35)N(B_2)+N(B_3)=0,$
 \item 2.4:  $N(D_2)=(\lambda_m-15)N(C_1)+N(C_2)=0,$
 \item 3.6: $N(D_3)=(\lambda_m-35)N(C_2)+N(C_3)=0,$
 \item 3.4: $N(E_3)=(\lambda_m-35)N(D_2)+N(D_3)=0.$
 \end{itemize}

\begin{equation}
\begin{array}{ccc}
& m\ \mbox{odd},\ n\ \mbox{even},\ k\ \mbox{odd} & \\[1.0ex]
& \left[\begin{array}{cccccccc}
3.3    & \phantom{-}3.5  & \phantom{-}3.13  & \phantom{-}3.11  &   &   &   & \\
3.5   & \phantom{-}2.3  &  \phantom{-}2.5 & \phantom{-}2.9  & \phantom{-}3.7  &   &   & \\
3.13   & \phantom{-}2.5  & \phantom{-}III  & \phantom{-}V  & \phantom{-}2.7  & \phantom{-}3.9  &   & \\
3.11    & \phantom{-}2.9  & \phantom{-}V  & \phantom{-}\bullet  & \phantom{-}\bullet  & \phantom{-}2.7  & \phantom{-}3.7  & \\
   & \phantom{-}3.7  & \phantom{-}2.7  & \phantom{-}\fbox{$\bullet$}  & \phantom{-}\fbox{$\bullet$} & \phantom{-}V  & \phantom{-}2.9  & \phantom{-}3.11\\
    &   & \phantom{-}3.9  & \phantom{-}2.7  & \phantom{-}V  & \phantom{-}III  & \phantom{-}2.5  & \phantom{-}3.13\\
   &   &   & \phantom{-}3.7  & \phantom{-}2.9  & \phantom{-}2.5  & \phantom{-}2.3  & \phantom{-}3.5 \\
    &   &   &   & \phantom{-}3.11  & \phantom{-}3.13  & \phantom{-}3.5  & \phantom{-}3.3\\
\end{array}\right] & 
\end{array},\notag
\end{equation}

\begin{itemize}
\item 2.7:  $N(A_1)=(m'+m+2)(\lambda_{m'}-\lambda_m+4)-\lambda_n(m'+1)=0,$
 \item $V$: $N(B_1) =(m'+1) [(m'+m+2)(m'+1)-\lambda_n]=0,$
 \item 3.7: $N(A_2)= (\lambda_{m'}-\lambda_m+16)N(A_1)-(\lambda_n-8)N(B_1)=0,$
 \item 2.9: $N(B_2)= (\lambda_{m'}+1)N(A_1)-(\lambda_n-8)N(B_1)=0,$
\item 3.11: $N(B_3)=(\lambda_{m'}+1)N(A_2)-(\lambda_n-24)N(B_2)=0,$
 \item $III$: $N(C_1)= (m'+1)[\lambda_m-\lambda_n-3+(m'+m+2)(m'+1)]=0,$
 \item 2.5: $N(C_2)=(\lambda_m-15)N(B_1)+N(B_2)=0,$
 \item 3.13: $N(C_3)=(\lambda_m-35)N(B_2)+N(B_3)=0,$
 \item 2.3: $N(D_2)=(\lambda_m-15)N(C_1)+N(C_2)=0,$
 \item 3.5: $N(D_3)=(\lambda_m-35)N(C_2)+N(C_3)=0,$
 \item 3.3: $N(E_3)=(\lambda_m-35)N(D_2)+N(D_3)=0.$
 \end{itemize}

For types 3.9 and 3.10, simultaneously consider the entries $Z_1$ of this type near $A_1$.  Application of both recurrences yields
\begin{equation}N(Z_1)=(\lambda_{m'}+\lambda_m+2)^2-4\lambda_{m'}\lambda_m-28-\lambda_n[(m'+1)N(A_0)+\lambda_m-3]\notag\end{equation}
and
\begin{equation}Z_1=\frac{A_0N(Z_1)}{(\lambda_{m'}-3)}.\notag\end{equation}

\section{Computer Implementation}
\label{sec:10}

In \cite{RRV}, an analysis is given for certain vanishing $c_{m, n, k}(i, j)$ with $J=m+n-k<3,000.$   We leave it to the reader to pursue those details there. 

The algorithms in this work require only rudimentary programming expertise, implemented on conventional hardware (2013 MacBook Pro).  The figures were constructed using Excel, which was also used to compute all $M(m, n, k)$ above.  Other algorithms, such the numerator formulas, were stress-tested using MAPLE.  

%%%%%%%%%%%%%%%%%%%%%%%% referenc.tex %%%%%%%%%%%%%%%%%%%%%%%%%%%%%%
% sample references
% %
% Use this file as a template for your own input.
%
%%%%%%%%%%%%%%%%%%%%%%%% Springer-Verlag %%%%%%%%%%%%%%%%%%%%%%%%%%
%
% BibTeX users please use
% \bibliographystyle{}
% \bibliography{}
%

\end{document}